\documentclass{amsart}
\usepackage{amsmath}
\usepackage{graphicx}
\usepackage{amsfonts}
\usepackage{amssymb}

\newtheorem{theorem}{Theorem}
\theoremstyle{plain}

\newtheorem{corollary}{Corollary}

\newtheorem{example}{Example}

\newtheorem{lemma}{Lemma}

\newtheorem{remark}{Remark}

\numberwithin{equation}{section}

\begin{document}
\title{The Classification of Generalized Riemann Derivatives}
\author{J. Marshall Ash}
\address{Department of Mathematics, DePaul University\\
Chicago, IL 60614}
\email{mash@condor.depaul.edu}
\urladdr{http://www.depaul.edu/\symbol{126}mash/}
\author{Stefan Catoiu}
\email{scatoiu@condor.depaul.edu}
\urladdr{http://www.depaul.edu/\symbol{126}scatoiu/}
\author{William Chin}
\email{wchin@condor.depaul.edu}
\urladdr{http://www.depaul.edu/\symbol{126}wchin/}
\thanks{This paper is in final form and no
version of it will be submitted for publication elsewhere.}
\date{v22: December 16, 2014}

\begin{abstract}
We characterize all pairs $(\Delta_{\mathcal{A}},\Delta_{\mathcal{B}})$ of generalized Riemann differences for which $\mathcal{A}$-differentiability implies $\mathcal{B}$-differentiability.
Two generalized Riemann derivatives $\mathcal{A}$ and $\mathcal{B}$ are equivalent if a function has a derivative in the sense of $\mathcal{A}$ at a real number $x$ if and only if it has a derivative in the sense of $\mathcal{B}$ at $x$. We determine the equivalence classes for this equivalence relation.
\end{abstract}

\maketitle


\section{Introduction}

\subsection{Motivation}
The $n$th Riemann difference of a function $f$ is the difference 
\begin{equation*}
\Delta_{n}f(x,h)=\sum_{k=0}^{n}(-1)^{k}\binom{n}{k}f(x+\left( n-k\right) h),
\end{equation*}%
and the $n$th symmetric Riemann difference of $f$ is 
\begin{equation*}
\Delta_{n}^{s}f(x,h)=\sum_{k=0}^{n}(-1)^{k}\binom{n}{k}f(x+(\frac{n}{2}-k)h).
\end{equation*}%
The function $f$ is $n$ times Riemann (resp. symmetric Riemann)
differentiable at $x$ if the limit $\displaystyle R_{n}f(x)=\lim_{h%
\rightarrow 0}\frac{\Delta_{n}f(x,h)}{h^{n}}\text{ (resp. $R_{n}^{s}f(x)=\lim_{h%
\rightarrow 0}\frac{\Delta_{n}^{s}f(x,h)}{h^{n}}$)}$ exists as a finite number.
If $f$ is $n$ times differentiable at $x$, one can use its $n$th Taylor polynomial about $x$ to see that $f$ is $n$
times Riemann and symmetric Riemann differentiable at $x$ and $%
R_{n}f(x)=R_{n}^{s}f(x)=f^{\left( n\right) }(x)$. With the exception of the $n=1$ forward Riemann case, where the definition of Riemann differentiation is the same as the one for ordinary differentiation, the converse of this is in general false. If $n\geq 2$, the function
\[ f(x)=\begin{cases}
0, &\text{if }x\in \mathbb{Q}\\
x, &\text{if }x\notin \mathbb{Q}
\end{cases}
\]
is $n$ times Riemann differentiable at $x=0$, without being first order differentiable at zero.
In the symmetric Riemann case, every discontinuous at zero odd function is symmetric Riemann differentiable at zero of all even orders without necessarily being differentiable at zero of any order, and every discontinuous at zero even function is symmetric Riemann differentiable at zero of all odd orders without necessarily being differentiable at zero of any order. Moreover,  symmetric Riemann differentiability of a certain order does not imply  symmetric Riemann differentiability of a different order. For example, at $x=0$ the even function
\[
f(x)=\begin{cases}
0, & \text{if } x\notin \mathbf{Q}\\
1, & \text{if } x\in \mathbf{Q}
\end{cases}
\]
is symmetric Riemann differentiable of all odd orders, and it is not symmetric Riemann differentiable of any even order. This example might lead one to guess that,
in the symmetric Riemann case, higher order differentiability implies lower order differentiability when the two orders have the same parity. This is disproved in greater generality in Section 6.1, Corollary 2.

Riemann derivatives were generalized in \cite{As}. A generalized $n$th
Riemann difference of a function $f$ is a difference of the form 
\begin{equation}
\Delta _{\mathcal{A}}f(x,h)=\sum_{i=1}^{m}A_{i}f(x+a_{i}h),
\label{GRD}
\end{equation}
where $\mathcal{A}=\{A_{1},\ldots ,A_{m};a_{1},\ldots ,a_{m}\}$ is a set of $2m$
parameters, with $a_{1},\ldots ,a_{m}$ distinct, whose elements satisfy the
Vandermonde conditions $\sum_{i=1}^{m}A_{i}a_{i}^{j}=n!\cdot \delta _{jn}$,
for $j=0,1,\ldots ,n$. By linear algebra, the excess number $e=m-\left( n+1\right) $ is non-negative. Some
interesting examples of $\mathcal{A}$-derivatives with positive excess appear in
numerical analysis; see \cite{AJ} and \cite{AJJ}. The \textit{$\mathcal{A}$-derivative} of $f$
is defined by the limit 
\begin{equation*}
D_{\mathcal{A}}f(x)=\lim_{h\rightarrow 0}\frac{\Delta _{\mathcal{A}}f(x,h)}{%
h^{n}}.
\end{equation*}%
For any $\mathcal{A}$, the derivative associated with $\mathcal{A}$ is a generalized Riemann derivative. Conversely, any generalized Riemann derivative is an $\mathcal{A}$-derivative for some $\mathcal{A}$. For simplicity, we will use the notation $\mathcal{A}$-derivative to mean both a particular and a generalized Riemann derivative. Context will always make clear which is meant. We have found this abuse of notation convenient. (E.g., the $\mathcal{A}$-derivatives $\mathcal{B}=\{1,-1;1,0\}$ and $\mathcal{C}=\{1,-1;1/2,-1/2\}$ are inequivalent, since $|x|$ has a $\mathcal{C}$-derivative at $x=0$, but not a $\mathcal{B}$-derivative there.)

Most of the results for classical Riemann derivatives hold true for $%
\mathcal{A}$-derivatives of differentiable functions $f$. For example, it is true that
\begin{equation}
\text{ordinary $n$th derivative exists at $x\Longrightarrow $ every $n$th $\mathcal{A}
$-derivative exists at $x$}  \label{imp}
\end{equation}%
and, as seen before, the converse of this is in general false for each $n$.

Our main motivation is the the following theorem of \cite{ACC} that, in the particular case of $n=1$, classifies all $\mathcal{A}$-derivatives for which the converse of (\ref{imp}) is true.

\begin{theorem}
\begin{description}
\item[A] The first order $\mathcal{A}$-derivatives which are dilates ($%
h\rightarrow sh$, for some $s\neq 0$) of%
\begin{equation}
\lim_{h\rightarrow 0}\frac{Af(x+rh)+Af(x-rh)-2Af(x)+f(x+h)-f(x-h)}{2h},
\label{Ar der}
\end{equation}%
where $Ar\neq 0$ are equivalent to ordinary differentiation.

\item[B] Given any other $\mathcal{A}$-derivative of any order $n=1,2,\dots $%
, there is a (Lebesgue) measurable function $f\left( x\right) $ such that $D_{\mathcal{A%
}}f\left( 0\right) $ exists, but the $n$th (Peano) derivative $f_{n}\left( 0\right) $
does not.
\end{description}
\end{theorem}

\begin{remark}
Part A of the above theorem displays some $\mathcal{A}$-derivatives
equivalent to the ordinary first order derivative. Part B asserts that no
first order $\mathcal{A}$-derivative that is not mentioned in Part A is equivalent to
ordinary first order differentiation and also that no higher order $%
\mathcal{A}$-derivative is equivalent to ordinary differentiation of the
same order. This makes Theorem 1 the best possible result with regard to
reversing the implication in (\ref{imp}). On the other hand, since the first
forward Riemann derivative is the same as the first ordinary derivative,
Theorem 1 classifies all $\mathcal{A}$-derivatives of order $1$ that are
equivalent to the Riemann derivative $R_{1}f(x)=\lim_{h\rightarrow 0}\frac{%
f(x+h)-f(x)}{h}$. This leads to the following more general problem: given any generalized Riemann derivative
$\mathcal{B}$, determine all generalized Riemann derivatives $\mathcal{A}$
such that 
\begin{equation}
\text{$\mathcal{B}$-derivative exists at $x\Longleftrightarrow $
$\mathcal{A}$-derivative exists at $x$.}  \label{A-B}
\end{equation}%
This is the main goal of the present work.
\end{remark}

\begin{remark}  The terms of the numerator in (\ref{Ar der}) fall into two categories: the sum of the
first three terms is a scalar multiple of an $r$-dilate of the even
difference $f(x+h)+f(x-h)-2f(x)$, and the remaining two terms add up to the odd
difference $f(x+h)-f(x-h)$. This motivates us to expect that a
classification of $\mathcal{A}$-derivatives given by the equivalence (\ref%
{A-B}) will be stated in terms of dilates, even differences, and odd
differences.
\end{remark} 

\subsection{Even and odd differences} A (not necessarily generalized Riemann) difference
\[
\Delta_{\mathcal{A}}f(x,h)=\sum A_if(x+a_ih)
\]
(where $\mathcal{A}=\{A_1,\ldots ,A_m;a_1,\ldots ,a_m\}$ is a set of $2m$ parameters with $a_1,\ldots ,a_m$ distinct)
is \textit{even} if $\Delta_{\mathcal{A}}f(x,-h)=\Delta_{\mathcal{A}}f(x,h)$ and it is \textit{odd} if $\Delta_{\mathcal{A}}f(x,-h)=-\Delta_{\mathcal{A}}f(x,h)$. For example, the symmetric Riemann difference $\Delta_nf(x,h)$ is even when $n$ is even and it is odd when $n$ is odd, while $\Delta_{\mathcal{A}}f(x,h)=2f(x+h)+f(x-h)-3f(x)$ is neither even nor odd.

Each difference $\Delta_{\mathcal{A}}$ gives rise to an even difference $\Delta_{\mathcal{A}}^{ev}$ and an odd difference $\Delta_{\mathcal{A}}^{odd}$, defined as
\begin{equation}
\begin{aligned}
\Delta_{\mathcal{A}}^{ev}f(x,h)&=\frac {\Delta_{\mathcal{A}}f(x,h)+\Delta_{\mathcal{A}}f(x,-h)}2=\sum_iA_i\frac { f(x+a_ih)+f(x-a_ih)}2  ,\\
\Delta_{\mathcal{A}}^{odd}f(x,h)&=\frac {\Delta_{\mathcal{A}}f(x,h)-\Delta_{\mathcal{A}}f(x,-h)}2=\sum_iA_i\frac { f(x+a_ih)-f(x-a_ih)}2 .
\end{aligned}
\label{evenalignedodd}
\end{equation}
The difference $\Delta_{\mathcal{A}}$ is even if and only if $\Delta_{\mathcal{A}}^{ev}=\Delta_{\mathcal{A}}$, and it is odd if and only if $\Delta_{\mathcal{A}}^{odd}=\Delta_{\mathcal{A}}$.
In addition, we have
\begin{equation}
\Delta_{\mathcal{A}}f(x,h)=\Delta_{\mathcal{A}}^{ev}f(x,h)+\Delta_{\mathcal{A}}^{odd}f(x,h).
\label{even-odd}
\end{equation}

Conversely, whenever $\Delta_{\mathcal{A}}$ is written as a sum $\Delta_{\mathcal{A}}=\Delta_{\mathcal{B}}+\Delta_{\mathcal{C}}$
of an even difference $\Delta_{\mathcal{B}}$ and an odd difference $\Delta_{\mathcal{C}}$, we must have
$\Delta_{\mathcal{B}}=\Delta_{\mathcal{A}}^{ev}$ and $\Delta_{\mathcal{C}}=\Delta_{\mathcal{A}}^{odd}$.
Relation (\ref{even-odd}) is therefore the unique writing of $\Delta_{\mathcal{A}}$ as a sum of an even difference and an odd difference, the \textit{components} of $\Delta_{\mathcal{A}}$.

\subsection{Results}

Our main result is the classification of all $\mathcal{A}$-derivatives given by the equivalence of generalized Riemann differentiation of (\ref{A-B}). Its statement can be written in a compact form by correlating parity of the order of the derivative and the parity of the component differences. For this, we define two maps
\[
\epsilon =\epsilon (n)=\begin{cases}
ev,& \text{if $n$ even}\\
odd, & \text{if $n$ odd}
\end{cases}
\text{ and }
\epsilon '=\epsilon '(n)=\begin{cases}
odd,& \text{if $n$ even}\\
ev, & \text{if $n$ odd}.
\end{cases}
\]
Given an $n$th generalized Riemann difference $\Delta _{\mathcal{B}}f(x,h)$, by the Vandermonde conditions, the difference
${\Delta _{\mathcal{B}}f(x,sh)}/{s^n}$ is the only scalar multiple of its $s$-dilate that is also an $n$th generalized Riemann difference. We call this a \textit{scaling} of $\Delta _{\mathcal{B}}f(x,h)$.
We have found the following complete classification of $\mathcal{A}$-derivatives.

\begin{theorem}
\label{2}Let $\mathcal{A}$ and $\mathcal{B}$ be generalized Riemann derivatives of orders $m$ and $n$. Then $%
\mathcal{B}$ is equivalent to $\mathcal{A}$ if and only if $m=n$ and there are non-zero constants $A,r$, and $s$ so that 
\begin{equation}
\frac{\Delta _{\mathcal{B}}f\left( x,h\right) }{h^{n}}=\frac{\Delta _{%
\mathcal{A}}^{\epsilon }f(x,sh)+A\Delta _{\mathcal{A}}^{\epsilon ^{\prime
}}f(x, rh )}{\left( sh\right) ^{n}}.  \label{Ar der1}
\end{equation}%
This means $\Delta _{\mathcal{B}}^{\epsilon }f(x,h)$ is a scaling of $\Delta _{\mathcal{A}}^{\epsilon }f(x,h)$ and $\Delta _{\mathcal{B}%
}^{\epsilon ^{\prime }}f(x,h)$ is a non-zero scalar multiple of a dilate of $%
\Delta _{\mathcal{A}}^{\epsilon ^{\prime }}f(x,h)$.
\end{theorem}

\begin{example}
\label{e:2}(i) Let $\Delta _{\mathcal{A}}=\Delta _{n}^{s}$ be the $n$th
symmetric Riemann difference. Theorem 2 says that the only generalized
derivatives that are equivalent to the $n$th symmetric Riemann derivative $%
R_{n}^{s}$ are the non-zero scalings of it.

(ii) By taking $\Delta _{\mathcal{A}}f(x,h)=\Delta _{1}f(x,h)=f(x+h)-f(x)$,
the first (forward) Riemann difference, the computation 
\begin{equation*}
\begin{aligned} \Delta_{\mathcal{A}}^{\epsilon
}f(x,h)=\Delta_{\mathcal{A}}^{odd }f(x,h)&=\frac
{[f(x+h)-f(x)]-[f(x-h)-f(x)]}2\\ &=\frac {f(x+h)-f(x-h)}2\\
\Delta_{\mathcal{A}}^{\epsilon '}f(x,h)=\Delta_{\mathcal{A}}^{ev
}f(x,h)&=\frac {[f(x+h)-f(x)]+[f(x-h)-f(x)]}2\\ &=\frac
{f(x+h)-2f(x)+f(x-h)}2 \end{aligned}
\end{equation*}%
shows that the classification in Theorem 2 is a generalization of the
classification in Theorem 1. 

(iii) The first order $\mathcal{A}$-derivative with excess $e=1$ given by%
\begin{equation*}
f_{\ast }\left( x\right) =\lim_{h\rightarrow 0}\frac{\left( \frac{1}{2}-\tau
\right) f\left( x+\left( \tau +1\right) h\right) +2\tau f\left( x+\tau
h\right) -\left( \frac{1}{2}+\tau \right) f\left( x+\left( \tau -1\right)
h\right) }{h},
\end{equation*}
where $\tau =1/\sqrt{3}$, arises in numerical analysis; see \cite{AJ}. A very
simple calculation done at the end of this paper using a convenient group algebra notation
developed in Section 5 shows that Theorem \ref{2} asserts that the most general
first order $A$-derivative equivalent to $f_{\ast }\left( x\right) $ is%
\begin{equation*}
\lim_{h\rightarrow 0}\frac{\Delta ^{odd}f\left(x, sh\right) }{sh}+A\frac{\Delta
^{ev}f\left(x, rh\right) }{h}
\end{equation*}%
where 
\begin{eqnarray*}
\Delta ^{odd}f\left(x, h\right)  &=\left( \frac{1}{2}-\tau \right) \frac{f\left( x+\left( \tau +1\right) h\right) -f\left( x-\left( \tau +1\right)h\right) }{2}+2\tau \frac{f\left( x+\tau h\right) -f\left( x-\tau h\right) }{2}\\
&-\left( \frac{1}{2}+\tau \right) \frac{f\left( x+\left( \tau -1\right)
h\right) -f\left( x-\left( \tau -1\right) h\right) }{2},
\end{eqnarray*}%
\begin{eqnarray*}
\Delta ^{ev}f\left(x, h\right)  &=\left( \frac{1}{2}-\tau \right) \frac{f\left( x+\left( \tau +1\right) h\right) +f\left( x-\left( \tau +1\right)h\right) }{2}+2\tau \frac{f\left( x+\tau h\right) +f\left( x-\tau h\right) }{2} \\
&+\left( \frac{1}{2}+\tau \right) \frac{f\left( x+\left( \tau -1\right)
h\right) +f\left( x-\left( \tau -1\right) h\right) }{2},
\end{eqnarray*}%
and $s$, $r$, and $A$ are nonzero constants.
\end{example}

We actually have a much more general theorem. For a given generalized Riemann derivative $\mathcal{B}$, it classifies all generalized Riemann derivatives $\mathcal{A}$ such that
\begin{equation}
\text{$\mathcal{B}$-derivative exists at $x\Longrightarrow $ $\mathcal{A}$-derivative exists at $x$.}  \label{A->B}
\end{equation}
The result is as follows:

\begin{theorem}
Let $\mathcal{A}$ and $\mathcal{B}$ be generalized Riemann derivatives of orders $m$ and $n$, respectively. Then $%
\mathcal{B}$-differentiation implies $\mathcal{A}$-differentiation if and only if $m=n$ and, for every function $f$,
$\Delta_{\mathcal{A}}^{\epsilon }f(x,h)$ and $\Delta_{\mathcal{A}}^{\epsilon '}f(x,h)$ are finite linear combinations
\[\begin{aligned}
\Delta_{\mathcal{A}}^{\epsilon }f(x,h)&=\sum_iU_i\Delta_{\mathcal{B}}^{\epsilon }f(x,u_ih)\\
\Delta_{\mathcal{A}}^{\epsilon '}f(x,h)&=\sum_iV_i\Delta_{\mathcal{B}}^{\epsilon '}f(x,v_ih)
\end{aligned}
\]
of non-zero
$u_i$-dilates of $\Delta_{\mathcal{B}}^{\epsilon }f(x,h)$ and $v_i$-dilates of $\Delta_{\mathcal{B}}^{\epsilon '}f(x,h)$.
\end{theorem}

A basic fact about ordinary derivatives is that the the existence of a lower order derivative does not imply the existence of a higher order derivative. The same is true about generalized Riemann derivatives. Specifically, the next example shows that the existence of a generalized Riemann derivative of a certain order does not imply the existence of any generalized Riemann derivative of higher order. This rules out the case $m>n$ in the above theorem.

\begin{example}
Fix $\mathcal{A}=\{A_{1},\ldots ;a_{1},\ldots \}$ of order $m$ and $\mathcal{B}=\{B_{1},\ldots ;b_{1},\ldots \}$ of order $n$, with $m>n$. We construct a function $f$ which is $\mathcal{B}$-differentiable and not $\mathcal{A}$-differentiable. Let $K$ be the the subfield of $\mathbb{R}$ generated over the rationals by all $a_i$'s and $b_i$'s, and define
\[
f(x)=\begin{cases}
x^m & \text{, if }x\in K\\
0 & \text{, if }x\notin K.
\end{cases}
\]
Note that by the definitions of $K$ and $f$, the operators $\lim_{h\rightarrow 0,h\in K}$ and $\lim_{h\rightarrow 0,h\notin K}$ act independently. They are easy to compute when applied to the quotients $\frac {\Delta_{\mathcal{A}}f(0,h)}{h^m}$ and $ \frac {\Delta_{\mathcal{B}}f(0,h)}{h^n}$. In the first case, these are the different numbers
\[\lim_{h\rightarrow 0,h\in K }\frac {\Delta_{\mathcal{A}}f(0,h)}{h^m}=\frac {d^m(x^m)}{dx^m}(0)=m!\;\text{ and }\;
\lim_{h\rightarrow 0,h\notin K}\frac {\Delta_{\mathcal{A}}f(0,h)}{h^m}=0,\]
so $D_{\mathcal{A}}f(0)$ does not exist. In the second case, the limits
\[\lim_{h\rightarrow 0,h\in K}\frac {\Delta_{\mathcal{B}}f(0,h)}{h^n}=\frac {d^n(x^m)}{dx^n}(0)=0\;\text{ and }\;
\lim_{h\rightarrow 0,h\notin K}\frac {\Delta_{\mathcal{B}}f(0,h)}{h^n}=0\]
are equal, so $D_{\mathcal{B}}f(0)$ exists. In both cases we used the fact that if the $k$th ordinany derivative of a function exists, then any $k$th generalized Riemann derivative exists and they are equal.
Since $K$ is countable, $f$ is measurable.
\end{example}

\begin{example} (i) By taking $\Delta_{\mathcal{B}}f(x,h)=\Delta_1f(x,h)=f(x+h)-f(x)$, the first Riemann difference we studied before, by Theorem 3 the first $\mathcal{A}$-derivatives implied by $R_1$ look like a first order linear combination of non-zero dilates of $R_1^s$ plus a linear combination of dilates of $R_2^s$. It is not hard to see that all these form the class of all first order $\mathcal{A}$-derivatives. Indeed, we already know that
``first order ordinary differentiable implies $\mathcal{A}$-differentiable, for every first order $\mathcal{A}$-derivative.''

(ii) By Theorem 2, there are no even first order and no odd second order $\mathcal{A}$-derivatives. More generally, for each $n$, there are no $n$th $\mathcal{A}$-derivatives of opposite parity.

By Theorem 3, the first symmetric Riemann derivative $R_1^s$ implies all odd first order $\mathcal{A}$-derivatives, and the second symmetric Riemann derivative $R_2^s$ implies all even second order $\mathcal{A}$-derivatives.
As we shall see next, these two implications do not extend to higher order symmetric Riemann derivatives.

(iii) Consider the difference
\[
\Delta_{\mathcal{A}}=\frac{f(2h)-2f(h)+2f(-h)-f(-2h)}{2}
\]
(the reader can check that this satisfies the Vandermonde conditions for a third order generalized Riemann difference) and let $\Delta_{\mathcal{B}}=\Delta_3^sf(0,h)$ be the third symmetric Riemann difference.
We claim that $\Delta_{\mathcal{A}}$ cannot be written as a linear combination of dilates of $\Delta_{\mathcal{B}}$ to deduce, by Theorem 3, that the generalized Riemann derivative corresponding to $\mathcal{A}$ is not implied by the symmetric Riemann derivative $R_3^s$.
To prove the claim, suppose $\Delta_{\mathcal{A}}$ is a linear combination of dilates of $\Delta_{\mathcal{B}}$. We write
\[
\Delta_{\mathcal{A}}=\sum_{i=1}^k\lambda_i\Delta_3^sf(0,r_ih)
=\sum_{i=1}^k\lambda_i\left[ f(\frac 32r_ih)-3f(\frac 12r_ih)+3f(-\frac 12r_ih)-f(-\frac 32r_ih)\right],
\]
where $\lambda_i\neq 0$, for all $i$. Since $\Delta_3^s$ is odd, all $r_i$'s may be taken to be positive, say $0<r_1<r_2<\cdots <r_k$. Note that the terms of the form $Af(rh)$ for the largest (resp. smallest) positive $r$ that appear in the expansions of both sides must be identical.
The largest of the dilates on each side appears only once, so $2=\frac 32r_k$; similarly, each smallest positive dilate appears only once, so $1=\frac 12r_1$.
In particular, $r_k=\frac 43$ is smaller than $r_1=2$, a contradiction.
\end{example}

More elegant equivalent formulations of the statements of Theorems 2 and 3 are given in Theorem 5, using the language of ideals of the group algebra $kG$, where the ground field $k=\mathbb{R}$ is the field of real numbers, and the group $G=\mathbb{R}^{\times }$ is the multiplicative group of non-zero real numbers. Then both theorems are proved in their reformulated forms.

A second kind of equivalence for generalized Riemann derivatives is a.e. equivalence. Say that two $n$th order generalized Riemann derivatives are \textit{a.e. equivalent} if for every measurable function, the set of all real $x$ where exactly one of them exists has measure $0$. In the 1930s, J. Marcinkiewicz and A. Zygmund proved that the Riemann derivative and the symmetric Riemann derivative are a.e. equivalent; see \cite{MZ}. In 1967 it was shown that \textit{all} $n$th order generalized Riemann derivatives are a.e. equivalent; see \cite{As}. This equivalence relation is as coarse as possible, having only one class for each degree. Until recently, the authors believed that the equivalence relation discussed in this paper was as fine as possible, each class consisting only of the scalings of a single $\mathcal{A}$-derivative. Theorem 1 corrected that false notion. Theorem 2 gives the exact answer.

\section{Orders of even and odd differences} 
Let $\Delta _{\mathcal{A}}f(x,h)=\sum_{i=1}^{m}A_{i}f(x+a_{i}h)$ be an $n$th generalized Riemann difference as defined in (\ref{GRD}). By (\ref{evenalignedodd}), the decomposition (\ref{even-odd}) can be written as
\begin{equation}
\begin{aligned}
\Delta_{\mathcal{A}}f(x,h)&=\sum_i B_i[f(x+b_ih)+f(x-b_ih)-2f(x)]\\
&+\sum_i C_i[f(x+c_ih)-f(x-c_ih)],
\label{2.1}
\end{aligned}
\end{equation}
where the $b_i$'s and $c_i$'s are all distinct and positive. The architecture of the above brackets automatically implies the $n$th Vandermonde condition $\sum_iA_i=0$ corresponding to $j=0$. The $n$th Vandermonde system then translates into the following:

\begin{equation}
\sum_iB_i[b_i^j+(-b_i)^j]+\sum_iC_i[c_i^j-(-c_i)^j]=n!\delta_{nj},
\label{BC}
\end{equation}
for $j=1,2,\ldots ,n$. If $S_1(j)$ and $S_2(j)$ are the above left and right sums, an immediate consequence of (\ref{BC}) is that
\begin{equation}
\begin{cases}
S_1(j)=0,&\text{if $j$ odd,}\\
S_2(j)=0,&\text{if $j$ even.}
\end{cases}
\label{S1S2}
\end{equation}
Assume that $n$ is an even number. Then (\ref{BC}) and (\ref{S1S2}) imply that
\begin{equation}
\begin{cases}
S_1(j)=n!\delta_{nj},&\text{if $j=2,4,...,n$,}\\
S_2(j)=0,&\text{if $j=1,3,...,n-1$.}
\end{cases}
\label{S1S2'}
\end{equation}
The systems (\ref{S1S2}) and (\ref{S1S2'}) together say that $\Delta_{\mathcal{A}}^{ev}$ satisfies the Vandermonde equations of order $n$, for $j=1,2,\ldots ,n$, and $\Delta_{\mathcal{A}}^{odd}$ satisfies Vandermonde equations of order higher than $n$, for $j=1,2,\ldots ,n$. With the Vandermonde equation for $j=0$ trivially satisfied for both component differences, we deduce that $\Delta_{\mathcal{A}}^{ev}$ is of order $n$ and $\Delta_{\mathcal{A}}^{odd}$ is either zero or of order $>n$. A similar discussion is conducted when $n$ is an odd number.

Let $\Delta _{\mathcal{A}}f(x,h)=\sum_{i=1}^{m}A_{i}f(x+a_{i}h)$ be a difference of a function $f$. By (\ref{even-odd}) and the notation of Section 1, we write $\Delta_{\mathcal{A}}f(x,h)=\Delta_{\mathcal{A}}^{\epsilon }f(x,h)+\Delta_{\mathcal{A}}^{\epsilon '}f(x,h)$ as a sum of its even and odd parts, where $\epsilon $ respects the parity of $n$ and $\epsilon '$ the opposite parity.

\begin{theorem}
The difference $\Delta_{\mathcal{A}}f(x,h)=\sum_{i=1}^{m}A_{i}f(x+a_{i}h)$ is a generalized $n$th Riemann difference if and only if the following two conditions hold:
\begin{enumerate}
\item[(i)] $\Delta _{\mathcal{A}}^{\epsilon }f(x,h)$ is a generalized Riemann difference of order $n$;
\item[(ii)] $\Delta _{\mathcal{A}}^{\epsilon '}f(x,h)$ is a scalar multiple of a generalized Riemann difference of order $>n$.
\end{enumerate}
\end{theorem}

\begin{proof}
One implication is already proved above. Conversely, suppose conditions (i) and (ii) are satisfied and that $n$ is an even number. By (\ref{S1S2'}), condition (i) translates into the Vandermonde system
\[
\sum_iA_i\frac {a_i^k+(-a_i)^k}2=n!\delta_{nk},
\]
for $k=0,1,\ldots ,n$, and condition (ii) translates into
\[
\sum_iA_i\frac {a_i^k-(-a_i)^k}2=0,
\]
for $k=0,1,\ldots ,n$.
We have
\[
\sum_iA_ia_i^k=\sum_iA_i\frac {a_i^k+(-a_i)^k}2+\sum_iA_i\frac {a_i^k-(-a_i)^k}2=n!\delta_{nk}+0=n!\delta_{nk},
\]
for $k=0,1,\ldots ,n$, so $\Delta _{\mathcal{A}}f(x,h)$ is an $n$th Riemann difference. The case when $n$ is odd is treated in a similar manner.
\end{proof}

\section{Dilations and the group algebra of $\mathbb{R}^{\times }$} The dilation by  a non-zero real number $r$ of the difference $\Delta_{\mathcal{A}}f(x,h)$ is the difference
\[
\Delta_{\mathcal{B}}f(x,h)=\sum A_if(x+a_irh).
\]
We write this as $x_r\cdot \Delta_{\mathcal{A}}=\Delta_{\mathcal{B}}$. Since
$x_1\cdot \Delta_{\mathcal{A}}=\Delta_{\mathcal{A}}$ and
$x_{rs}\cdot \Delta_{\mathcal{A}}=x_r\cdot (x_s\cdot \Delta_{\mathcal{A}})$, dilation is a group action of the multiplicative group of non-zero real numbers $\mathbb{R}^{\times }$ on the real vector space of all differences of $f$.
On the other hand, since a linear combination with real coefficients of dilates of a difference is also a difference, this gives an action of the real vector space $\text{span}_{\mathbb{R}}\{x_r\;|\;r\in \mathbb{R}^{\times }\}$ on the same space of differences. For example,
\[
(3x_2-ex_{\sqrt{3}})\cdot \Delta_{\mathcal{A}}=3(x_2\cdot \Delta_{\mathcal{A}})-e(x_{\sqrt{3}}\cdot \Delta_{\mathcal{A}}).
\]
The group action and the vector space action together give an action of the group algebra $k\mathbb{R}^{\times }$ of the multiplicative group of the real numbers over the real field $k=\mathbb{R}$. This is the vector space
\[
\mathbf{A}=k\mathbb{R}^{\times }=\text{span}_{k}\{x_r\;|\;r\in \mathbb{R}^{\times }\},
\]
with the multiplication of basis elements given by
\[
x_{r}x_{s}=x_{rs}\text{, for all $r,s\in \mathbb{R}^{\times }$.}
\]
For example, in $\mathbf{A}$ we have \[(2x_1-7x_{-5})(\sqrt{2}x_3-x_{\pi })=2\sqrt{2}x_3-7\sqrt{2} x_{-15}-2x_{\pi }+7x_{-5\pi }.\]

The above discussion can be formalized by saying that the space $D=D(f,x,h)$ of all differences of a function $f$ at $x$ and $h$ without $h$-constant term becomes an $\mathbf{A}$-module via the map $\mathbf{A}\times D\rightarrow D$, given by
\[(x_r,f(x+ah))\mapsto x_r\cdot f(x+ah):=f(x+rah)\text{, for all }r,a\in \mathbb{R}^{\times }.\]
We observe that the unique linear map $\theta :\mathbf{A}\rightarrow D$, defined by $\theta (x_a)=f(x+ah)$, for all $a\in \mathbb{R}^{\times }$, is actually an onto $\mathbf{A}$-module map. Indeed, $\theta (x_rx_a)=\theta (x_{ra})=f(x+rah)=x_r\cdot f(x+ah)=x_r\cdot \theta (x_{a})$. An easy application of this is that
\[D=\theta (\mathbf{A})=\theta (\mathbf{A}x_1)=\mathbf{A}\cdot \theta (x_1)=\mathbf{A}\cdot f(x+h),\]
so $D$ is a cyclic $\mathbf{A}$-module generated by $f(x+h)$.
Indeed, each difference $\Delta_{\mathcal{A}}f(x,h)=\sum_iA_if(x+a_ih)$ of $f$ without $h$-constant term $Af(x)$ can be written as the action
\begin{equation}
\Delta_{\mathcal{A}}f(x,h)=(\sum_iA_ix_{a_i})\cdot f(x+h).
\label{orbit}
\end{equation}
For more properties of group algebras see \cite{P}. 

We also define the algebra $\mathbf{B}$ to be the $k$-semigroup algebra of the multiplicative semigroup (actually monoid) of all real numbers, by adjoining the absorbing basis element $x_0$. Its multiplication is given by $x_0x_r=x_0$, for all real numbers $r$. Since
\[ x_0\cdot f(x+h)=f(x),\]
in light of (\ref{orbit}), the space of all differences of $f$ at $x$ and $h$ is a cyclic $\mathbf{B}$-module generated by $f(x+h)$. Being more inclusive, the action of $\mathbf{B}$ is more general than the action of its subalgebra $\mathbf{A}$. On the other hand, the difference between $\mathbf{B}$ and $\mathbf{A}$ is just the basis element $x_0$. This translates into the space of all differences of $f$ at $x$ and $h$ splitting into differences that do and those that do not contain $f(x)$ as a term. Our concern being with $n$th $\mathcal{A}$-differences of $f$, the only adjustment we make between the actions of algebras $\mathbf{A}$ and $\mathbf{B}$ on $f(x+h)$ is at the Vandermonde conditions for $j=0$. See the sentence following equation (\ref{2.1}) to understand how the $j=0$ condition naturally fits into the theory.

We shall see that questions about generalized derivatives translate into questions about principal ideals of $\mathbf{A}$.

\section{Properties of the groups $\mathbb{R}^{+}$ and $\mathbb{R}^{\times }$}
The multiplicative group $\mathbb{R}^{\times }$ of non-zero real numbers has a single torsion element $-1$ of order two. \textit{Torsion elements} in a group are its elements of finite order. A \textit{torsion group} is a group whose elements are all of finite order. 
The group isomorphism $r\mapsto (\text{sign}(r),|r|)$ gives a decomposition of $\mathbb{R}^{\times }$ as a direct product $\langle -1\rangle \times \mathbb{R}^+$ of its torsion subgroup $\langle -1\rangle $ and its torsion-free subgroup $\mathbb{R}^+$.

We digress for a paragraph and recall coproducts and direct sums. The \textit{direct product} of a family of groups $\{G_i\}_{i\in I}$ is their Cartesian product
\[\prod_{i\in I}G_i=\{(g_i)_{i}\;|\;g_i\in G_i,i\in I\},\]
with the componentwise multiplication. Their \textit{coproduct} or \textit{direct sum} is the subgroup \[\coprod G_i=\bigoplus_{i\in I}G_i=\{(g_i)_{i}\;|\;g_i=1_{G_i}\text{, for all but finitely many }i\in I\}.\]
The above two definitions coincide precisely when the indexing set $I$ is finite. When dealing with multiplicative groups, we prefer the coproduct $\coprod $ because the direct sum notation $\oplus $ connotes additivity. The direct product enjoys the uniqueness of expression of elements by their components, while the coproduct additionally requires the finiteness of expressions. The usual notion of a direct sum of vector spaces is an abelian group direct sum with an additional scalar structure. The finiteness condition comes from the property that every element of a vector space is a finite linear combination of basis elements.

We show that the multiplicative group $\mathbb{R}^{+}$ is the coproduct of $2^{\aleph_0}$ copies of $\mathbb{Q}$. A way to see this is to first notice that $\log :\mathbb{R}^{+}\rightarrow \mathbb{R}$ is a group isomorphism from the multiplicative group $\mathbb{R}^{+}$ to the additive group $\mathbb{R}$ (with inverse map $\exp :\mathbb{R}\rightarrow \mathbb{R}^{+}$), and observe that $\mathbb{R}$ is a direct sum of copies of $\mathbb{Q}$. Indeed, the field $\mathbb{R}$ is a vector space over its subfield $\mathbb{Q}$, hence it has a basis, say $\{ \log \lambda_i\}_{i\in I}$ (known as a Hamel basis). So we see that $\mathbb{R}$ is a direct sum of the $2^{\aleph_0}$ one-dimensional subspaces generated by the basis elements. The subspaces are each isomorphic to $\mathbb{Q}$ as abelian groups.

In fact, it is apparent that $\log :\mathbb{R}^{+}\rightarrow \mathbb{R}$ is also a $\mathbb{Q}$-vector space isomorphism, where the scalar multiplication on the multiplicative abelian group $\mathbb{R}^{+}$ is defined by exponentiation. So now $\{\lambda_i\}_{i\in I}$ is a basis of $\mathbb{R}^{+}$ over $\mathbb{Q}$, and the span of each $\lambda_i$ is written $\lambda_i^{\mathbb{Q}}$. Thus we have the $\mathbb{Q}$-vector space decomposition
\[
\mathbb{R}^{+}=\exp (\sum \mathbb{Q} \log \lambda_i)=\coprod \lambda_i^{\mathbb{Q}}.
\]
Equivalently, every element of $\mathbb{R}^{+}$ is uniquely expressible as a finite product $\prod \lambda_i^{q_i}$, with $q_i\in \mathbb{Q}$.

\begin{remark}
The decomposition of $\mathbb{R}^{+}$ also follows from the structure theory for abelian torsion-free divisible groups; see \cite{F} or \cite{K}. This involves the basic fact that abelian divisible groups are injective $\mathbb{Z}$-modules, a core concept of homological algebra; see \cite{AF} or \cite{CE}.
\end{remark}

We consider the algebra of generalized polynomials $k[x]_{\mathbb{Q}}$ where rational exponents are allowed. One can define $k[x]_{\mathbb{Q}}$ as the union
\[
\bigcup_{n>0}k[x^{\pm \frac 1n} ]
\]
of Laurent polynomial rings using formal $n$th roots of the indeterminate $x$. Algebras of generalized polynomials $k[x_i\;|\;i\in I]_{\mathbb{Q}}$ in many variables $x_i$ are defined similarly. These algebras have bases of generalized monomials given by finite products $\prod x_i^{q_i}$ with $q_i\in \mathbb{Q}$ where the multiplication is the obvious one, agreeing with the ones in the polynomial subrings $k[x_i^{\pm \frac 1n} \;|\;i\in I]$.

\begin{lemma}
The group algebra $k\mathbb{R}^{+}$ is a generalized polynomial algebra  over $k$ in continuum-many variables.
\end{lemma}

\begin{proof}
Define a map $\nu :k[x_i\;|\;i\in I]_{\mathbb{Q}}\rightarrow k\mathbb{R}^{+}$ on monomials by
\[
\nu (\prod x_i^{q_i})=x_{\prod\lambda_i^{q_i}}
\]
and extend $k$-linearly. By the discussion above concerning the uniqueness for the direct sum decomposition, $\nu $ is a $k$-linear bijection. It also follows directly from the definition of group algebra $k\mathbb{R}^{+}$ that  $\nu $ is an algebra homomorphism. This completes the proof of the lemma.
\end{proof}

\section{The algebras $\mathbf{A}$, $\mathbf{B}$ and their associated differences}

This section is useful in getting familiar with the basic properties of the algebras $\mathbf{A}$ and $\mathbf{B}$ and their interpretation with differences.  It is essential to understanding the rest of the paper, though non-algebraists may wish to skim it on the first pass.

\subsection{Properties of the algebras $\mathbf{A}$ and $\mathbf{B}$.} The group algebra $\mathbf{A}=k\mathbb{R}^{\times }$ and the semigroup algebra $\mathbf{B}=k\mathbb{R}$ have many basic properties some of which translating into properties of differences of a function $f$. Here are some of these properties.

\begin{enumerate}
\item[(AB1)] Both $\mathbf{A}$ and $\mathbf{B}$ have the same identity element $1=x_1$.
\item[(AB2)] Let $\sigma =x_{-1}$. The elements $\pm \sigma $ are the only elements of $\mathbf{A}$ and $\mathbf{B}$ with multiplicative order $2$.
\item[(AB3)] $\mathbf{A}$ contains two orthogonal idempotents $e=\frac 12(1+\sigma )$ and $d=\frac 12(1-\sigma )$ such that $e+d=1=x_1$. These are elements that satisfy $e^2=e$, $d^2=d$, and $de=0$. Pairwise orthogonal idempotents that add up to one are responsible for the decomposition of an algebra as a direct sum of ideals that are also unital subalgebras. In our case,
\begin{equation}
\mathbf{A}=e\mathbf{A}\oplus d\mathbf{A},
\label{eA+dA}
\end{equation}
where $e$ is the identity of $e\mathbf{A}$ and $d$ is the identity of $d\mathbf{A}$.
Algebra $\mathbf{B}$ contains the absorbing idempotent $x_0$, and $\mathbf{B}=kx_0\oplus \mathbf{A}$, a direct sum of unital subalgebras, not ideals. Note that $ex_0=x_0$ and $dx_0=0$. Therefore
\begin{equation}
\mathbf{B}=(kx_0\oplus e\mathbf{A})\oplus d\mathbf{A}=e\mathbf{B}\oplus d\mathbf{B}.
\label{eB+dB}
\end{equation}
\item[(AB4)] Let $e_r=\frac 12(x_r+x_{-r})=ex_r\in e\mathbf{A}$ and $d_r=\frac 12(x_r-x_{-r})=dx_r\in d\mathbf{A}$. \textit{E.g.} we have $d_1=d$ and $e_1=e$. An element of $\mathbf{A}$ looks like $\alpha =\sum_r A_rx_{r}$ and a generic element of $e\mathbf{A}$ is $e\alpha =\sum_r A_re_{r}$. Since $e_{-r}=e_r$, we can write $e\alpha =\sum_{r>0} (A_r+A_{-r})e_{r}$, so
\[e\mathbf{A}=\sum_{r>0} ke_{r}\text{ and similarly }d\mathbf{A}=\sum_{r>0} kd_{r}.\]
\item[(AB5)] After relabeling of coefficients, the above generic element of $e\mathbf{A}$ has the form of a finite sum indexed by positive real numbers
\[
e\alpha =\sum_{r>0} A_re_{r}=\sum_{r>0} A_r\frac {x_r+x_{-r}}2.
\]
Its expression is uniquely determined by twice its positive part $\sum_{r>0} A_rx_r$.
The structure of $e\mathbf{A}$ is then revealed by the mapping $k\mathbb{R}^+\rightarrow e\mathbf{A} $ given by $x_r\mapsto e_r$, for $r>0$. One easily checks that this map is an algebra isomorphism. The result of this is that $e\mathbf{A}$ and (similarly) $d\mathbf{A}$ are isomorphic to the group algebra $k\mathbb{R}^+$. By Lemma 1, we see that both $e\mathbf{A}$ and $d\mathbf{A}$ are generalized polynomial algebras in $2^{\aleph_0}$ variables. This structure will be used in Example 4.
\item[(AB6)] Equation (\ref{eA+dA}) is the decomposition of $\mathbf{A}$ as a direct sum of its eigenspaces given by multiplication by $\sigma $. Specifically, $e\mathbf{A}$ is the $\sigma $-eigenspace with eigenvalue $+1$ and $d\mathbf{A}$ is the $\sigma $-eigenspace with eigenvalue $-1$. We call $e\mathbf{A}$ the even part of $\mathbf{A}$, and $d\mathbf{A}$ is the odd part of $\mathbf{A}$. 
\item[(AB7)] The elements of $\mathbf{A}$ of the form $cx_r$, where $c,r\in\mathbb{R}^{\times }$ are invertible in $\mathbf{A}$. These are the \textit{trivial units} of $\mathbf{A}$; their inverses are $c^{-1}x_{r^{-1}}$. We shall see next that the trivial units are not all of the units of $\mathbf{A}$.
It is well-known (see \cite{P}, Chapter 13) that the group algebra of a torsion-free abelian group has only trivial units. In particular, the same is true for the group algebra $k\mathbb{R}^+$ and by (AB5) for its isomorphic copies $e\mathbf{A}$, and $d\mathbf{A}$. The (trivial) units in $e\mathbf{A}$ are $e$-multiples of trivial units in $\mathbf{A}$. This means they are of the form $Aex_r=Ae_r$, for $0\neq A\in k$ and $r>0$. By the same token, the units of $d\mathbf{A}$ are of the form $Bdx_s=Bd_s$, for $0\neq B\in k$ and $s>0$. Since the units of a direct sum of algebras are sums of the units of the summands, by (\ref{eA+dA}) we conclude that the units of $\mathbf{A}$ are the elements of the form
\[ Ae_r+Bd_s,\]
for $0\neq A,B\in k$ and $r,s>0$.
\item[(AB8)] Let $V=\sum_{r\in k}kx_r$. Then the $k$-space of all real functions
\[\mathcal{F}(k)=\{f\;|\;f:k\rightarrow k\}\] is isomorphic to the dual $k$-space
\[V^{*}=\text{Hom}_k(V,k)=\{\varphi \;|\;\varphi :V\rightarrow k\}\]
via the mapping $\mathcal{F}(k)\ni f\mapsto \varphi \in V^{*}$, where $\varphi $ is the linear map defined on basis elements of $V$ by $\varphi (x_r)=f(r)$, for all $r\in k$.
\item[(AB9)] This property, which looks more technical than it really is, is only needed in the proof of Corollary 1.
\newline Let $k'$ be a subfield of $k$ and $V'=\sum_{r\in k'}k'x_r$. Each function $f:k'\rightarrow k'$ can be viewed as a function $f:k\rightarrow k$ by setting $f(t)=0$, for $t\in k\setminus k'$. In this way, the function space $\mathcal{F}(k')$ embeds naturally in $\mathcal{F}(k)$ as the $k'$-subspace
\[\mathcal{F}'(k)=\{f\in \mathcal{F}(k)\;|\;R(f)\subseteq k'\text{ and }k\setminus k'\subseteq N(f)\},\]
where $R(f)$ and $N(f)$ are the range and the nullset of $f$.
By (AB8) applied to $k'$ and $V'$ in place of $k$ and $V$, we have $\mathcal{F}(k')$ is $k'$-linear isomorphic to ${V'}^{*}=\text{Hom}_{k'}(V',k')$.
Let $\mathcal{S}=\{1\}\cup \mathcal{T}$ be a basis of $k$ over $k'$. Then $\bigcup_{r\in k}\mathcal{S}x_r$ is a $k'$-basis of $V$.
Each $k'$-linear map $\varphi :V'\rightarrow k'$ naturally extends to a $k'$-linear map $\varphi :V\rightarrow k$ by setting $\varphi(t)=0$, for $t\in \left(\bigcup_{r\in k}\mathcal{S}x_r\right)\setminus \{x_r|r\in k'\}$, which is also an element of $V^{*}$. In this way, ${V'}^{*}$ is $k'$-linearly isomorphic to
\[V^{*'}=\{\varphi \in V^{*}\;|\;R(\varphi )\subseteq k'\text{ and }(\bigcup_{r\in k}\mathcal{S}x_r)\setminus \{x_r|r\in k'\}\subseteq N(\varphi )\}\]
which is a $k'$-subspace of $V^{*}$. Moreover, $\mathcal{F}'(k)$ is isomorphic to $V^{*'}$ as a $k'$-space via the mapping $\mathcal{F}'(k)\ni f\mapsto \varphi \in V^{*'}$, where $\varphi $ is defined on $k'$-basis elements of $V$ by $\varphi (x_r)=f(r)$ and $\varphi (\mathcal{T}x_r)=0$, for all $r\in k$.
\end{enumerate}

\[
\begin{array}{ccccc}
\mathcal{F}(k') &\rightarrow &\mathcal{F'}(k)&\hookrightarrow &\mathcal{F}(k)\\
\downarrow & & \downarrow & & \downarrow \\
{V'}^{*}&\rightarrow &V^{*'}&\hookrightarrow & V^{*} 
\end{array}
\]
The above diagram illustrates the relationship between the vector spaces and linear maps considered in (AB8) and (AB9). The right vertical arrow is the $k$-linear isomorphism of (AB8). The four arrows to the left are the $k'$-linear isomorphisms of (AB9). The two hook-arrows are inclusions of $k'$-subspaces.

\subsection{And their interpretation with differences.} The above properties translate into the language of even and odd differences of Section 1.2 or group algebra actions of Section 3. For example, we have
\[
\begin{aligned}
e\cdot f(x+h)&=\frac {f(x+h)+f(x-h)}2,\\
d\cdot f(x+h)&=\frac {f(x+h)-f(x-h)}2.
\end{aligned}
\]
More generally, if $\alpha =\sum_iA_ix_{a_i}$ is the group algebra element determined by $\Delta_{\mathcal{A}} $ in (\ref{orbit}), then
\[
\begin{aligned}
(e\alpha )\cdot f(x+h)=e\cdot \Delta_{\mathcal{A}}f(x,h)&=\Delta_{\mathcal{A}}^{ev}f(x,h),\\
(d\alpha )\cdot f(x+h)=d\cdot \Delta_{\mathcal{A}}f(x,h)&=\Delta_{\mathcal{A}}^{odd}f(x,h).
\end{aligned}
\]
The second equalities say that the actions of $e$ and $d$ on the space $\mathcal{D}$ of all differences of $f$ at $x$ and $h$ map these differences onto their even and odd parts, respectively. For the first equalities, we notice that $\alpha \in \mathbf{A}$ (resp. $\alpha \in \mathbf{B}$) is equivalent to $\Delta_{\mathcal{A}}f(x,h)$ is a difference without $f(x)$-term (resp. $\Delta_{\mathcal{A}}f(x,h)$ is any difference of $f$). We deduce that the submodules of $\mathcal{D}$ generated by $f(x+h)$ under the actions of $e\mathbf{A}$ and $d\mathbf{A}$ (resp. $e\mathbf{B}$ and $d\mathbf{B}$) are exactly all even and odd differences without $f(x)$-term (resp. all even and odd differences of $f$). Note that odd differences do not have $f(x)$-terms, so the actions of $d\mathbf{A}$ and $d\mathbf{B}$ on $f(x+h)$ coincide. The natural conclusion from this discussion is that the actions of $\mathbf{A}$ and $\mathbf{B}$ on $f(x+h)$, which are sums of the actions of their components given in (\ref{eA+dA}) and (\ref{eB+dB}), amount to splitting their corresponding differences into even and odd components given in (\ref{even-odd}).

\section{Translation Theorem and Proof of Main Results}

The major result of this section, Theorem 5, translates the implication and the equivalence of generalized derivatives into the inclusion and equality of principal ideals of algebras $\mathbf{A}$ or $\mathbf{B}$. We then prove the main classification results of Theorems 2 and 3 stated in the Introduction by means of classifying principal ideals of $\mathbf{A}$ or $\mathbf{B}$ by inclusion or equality. We prove this easily by reference to the basic group algebra properties (AB1)-(AB9) of Section 5.

\subsection{Translation Theorem.} Let $\alpha =\sum_iA_ix_{a_i}$ and $\beta =\sum_iB_ix_{b_i}$ be the elements of $\mathbf{B}$ that correspond to the differences $\Delta_{\mathcal{A}}f(x,h)=\sum_iA_if(x+a_ih)$ and $\Delta_{\mathcal{B}}f(x,h)=\sum_iB_if(x+b_ih)$, as defined in Section 5.2. Call $\{ x_{b_1},x_{b_2},\ldots \}$ the \textit{support} of $\beta $.
As examples, $\alpha =d_1$ corresponds to the first symmetric Riemann difference $\Delta_1^sf(x,h)$, $\alpha =x_1-x_0$ corresponds to the first (forward) Riemann difference $\Delta_1f(x,h)$, and $\alpha =A(e_r-x_0)+d_1$ corresponds to the first difference $\Delta_{\mathcal{A}}f(x,h)$ of Theorem 1.

The ideal of $\mathbf{B}$ generated by $\alpha $ is denoted by $(\alpha )$. Let $\alpha_r =\alpha x_r=\sum A_ix_{a_ir}$ be the \textit{dilate} of $\alpha $ by $r\in k^{\times }$. Note especially that $x_r$ invertible implies $(\alpha )=(\alpha_r)$, and since $(\alpha )$ is the set of all $\mathbf{B}$-multiples of $\alpha $, the ideal $(\alpha )$ is the linear span of all dilates of $\alpha $.
Moreover, since in arguments about $\mathcal{A}$-differentiability of a general function $f$ at $x$ we can always assume without loss that $x=0$ and $f(0)=0$, this amounts to factoring out the ideal $(x_0)$ of $\mathbf{B}$, that is to projecting down to $\mathbf{A}$. In this process, the ideals of $\mathbf{B}$ may be assumed to be ideals of $\mathbf{A}$.

\begin{theorem} With the above notation, if $\alpha $ and $\beta $ correspond to order $n$ generalized Riemann differences $\Delta_{\mathcal{A}}$ and $\Delta_{\mathcal{B}}$, then
\begin{enumerate}
\item[(i)] $(\alpha )\supseteq (\beta )$ iff ``$\mathcal{A}$-differentiable at $x$ $\Longrightarrow $ $\mathcal{B}$-differentiable at $x$''.
\item[(ii)] $(\alpha )= (\beta )$ iff ``$\mathcal{A}$-differentiable at $x$ $\Longleftrightarrow $ $\mathcal{B}$-differentiable at $x$''.
\end{enumerate}
\end{theorem}

\begin{proof}
(i) Assume $(\alpha )\supseteq (\beta )$. We write $\beta $ as a finite sum $\beta=\sum_rc_r\alpha_r$ with $r,c_r\in k$. Let $f$ be $\alpha $-differentiable at $x$ and denote $D_{\alpha }f(x)=d\in k$. As an abuse of notation, we write $\alpha $-differentiable, $D_{\alpha }f(x)$ and $\alpha (x,h)$ to respectively denote $\mathcal{A}$-differentiable, $D_{\mathcal{A}}f(x)$ and $\Delta_{\mathcal{A}}f(x,h)$. Since
\[
D_{\alpha_r }f(x)=\lim_{h\rightarrow 0}\frac {\alpha_r(x,h)}{h^n}=\lim_{h\rightarrow 0}\frac {\alpha(x,rh)}{(rh)^n}\cdot r^n=r^nD_{\alpha }f(x),
\]
linearity of the limit operator makes $f$ a $\beta$-differentiable function at $x$ and $D_{\beta }f(x)=\sum_rc_rr^nd$. This sum actually equals $d$, since $\beta $ is a generalized Riemann derivative.

It remains to prove the converse. Let $G$ be the subgroup of $k^{\times }$ generated by all non-zero $a_i$'s and $b_i$'s. Then $G$ is countable while the set of its cosets in $k^{\times }$ is not. Consider $\{s_n\}_{n>0}$ be a sequence of representatives of cosets of $G$ in $k^{\times }$ such that $\lim_{n\rightarrow \infty }s_n=0$. We prove the contrapositive statement: assuming $(\alpha )\nsupseteq (\beta )$, we show that there exists a function $f$ such that $D_{\mathcal{A}}f(0)$ exists but $D_{\mathcal{B}}f(0)$ does not exist.

Observe that, by assumption and the obvious fact that $(\beta_{s_n})=(\beta )$, we have $\beta_{s_n}\notin (\alpha )$ for all $n$. Moreover, the $\beta_{s_n}$'s are linearly independent modulo the ideal $(\alpha )$. To see this, we note that the support of $\beta  $ is included in $G$ and this makes the supports of the $\beta_{s_n}$'s included in the $Gs_n$'s, hence they are pairwise disjoint. Suppose that $\sum \lambda_n\beta_{s_n}\in (\alpha )$, for $\lambda_n\in k$. We write this as $\sum \lambda_n\beta_{s_n}=\alpha \sum \mu_ix_{r_i}$, for $\mu_i\in k$ and $r_i\in k^{\times}$. For each $n$, let $t_{s_n}$ be the sum of all terms $\mu_ix_{r_i}$ of the last sum for which $r_i\in G{s_n}$, and let $t'$ be the sum of the remaining terms. The last equation becomes $\sum \lambda_n\beta_{s_n}=\alpha (t'+\sum t_{s_n})=\alpha t'+\sum \alpha t_{s_n}$. The expression of an element of a group algebra as a sum of elements with supports in distinct cosets of a subgroup is unique.
Thus $\lambda_n\beta_{s_n}=\alpha t_{s_n}$, for all $n$, and $0=\alpha t'$. If $\lambda_n\neq 0$, for some $n$, then
$\beta_{s_n}\in (\alpha )$, a contradiction.

Let $V$ be the vector space defined in (AB8) of Section 5.1. The axiom of choice implies that every linearly independent subset of $V$ can be completed to a basis, so let $W$ be a complement of the subspace $(\alpha )\oplus \sum k\beta_{s_n}$ in $V$. This means that
\[
V=(\alpha )\oplus \sum k\beta_{s_n}\oplus W.
\]
Define a functional $\varphi \in V^{*}$ by setting $\varphi $ identically equal to zero on both $(\alpha )$ and $W$, and $\varphi (\beta_{s_n})=1$, for all $n$. Then the corresponding function $f$ has
\[ \Delta_{\mathcal{A}}f(0,h)=\sum A_if(a_ih)=\sum A_i\varphi (x_{a_ih})=\varphi \left(\sum A_ix_{a_ih}\right)=\varphi (\alpha_h)=0,\]
since $\alpha_h\in (\alpha )$, for all $h$. Thus $D_{\mathcal{A}}f(0)=\lim_{h\rightarrow 0}\frac {\Delta_{\mathcal{A}}f(0,h)}{h^n} =0$. On the other hand,
\[ \Delta_{\mathcal{B}}f(0,s_m)=\sum B_if(b_is_m)=\sum B_i\varphi (x_{b_is_m})=\varphi \left(\sum B_ix_{b_is_m}\right)=\varphi (\beta_{s_m})=1,\]
for all $m$, implies that $D_{\mathcal{B}}f(0)=\lim_{m\rightarrow \infty }\frac {\Delta_{\mathcal{B}}f(0,s_m)}{s_m^n} =\lim_{m\rightarrow \infty }\frac {1}{s_m^n} $ does not exist as a finite number.

Part (ii) is an easy consequence of part (i).
\end{proof}

\begin{corollary}
The counterexample $f$ constructed in the proof of Theorem 5 may not be a measurable function. Nevertheless, the proof can be adapted to make $f$ measurable.
\end{corollary}

\begin{proof} We follow closely the proof of Theorem 5 and stress only the differences. Let $k'$ be the subfield of $k$ generated over the rationals by all $A_i,a_i,B_i,b_i$ and $\{s_n\}_{n\geq 1}$. This is a countable field. Recall from earlier in the section that the ideal $(\alpha )$ is the $k$-span of all shifts $\alpha_r=\alpha x_r$, for $r\in k$. We define $[\alpha ]$ to be the $k'$-span of all shifts $\alpha_r$, for $r\in k'$. Then $(\alpha )=([\alpha ])$ and also $(\beta )=([\beta ])$. Note that the $\beta_{s_n}$'s are $k'$-linearly independent modulo $[\alpha ]$, since they are $k$-linearly independent modulo $(\alpha )$. Let $W$ be a $k'$-complement of $[\alpha ]\oplus \sum k'\beta_{s_n}$ in $V$. This means that
\[
V=[\alpha ]\oplus \sum k'\beta_{s_n}\oplus W.
\]
Since the first two terms above are part of $V'$, using the $k'$-basis of $V$ of (AB9), we may assume that $W$ contains all basis elements of $V$ that are not basis elements of $V'$.
Define a functional $\varphi \in \text{Hom}_{k'}(V,k)$ by setting $\varphi $ identically equal to zero on both $[\alpha ]$ and $W$, and $\varphi (\beta_{s_n})=1$, for all $n$. Then $\varphi $ is an element of $V^{*'}$. Let $f\in \mathcal{F}'(k)$ be the function  that corresponds to $\varphi $ via the last isomorphism in (AB9). Let $h$ be any real number. If $h\in k'$, then $\alpha_h\in [\alpha ]$, so
\[ \Delta_{\mathcal{A}}f(0,h)=\sum A_if(a_ih)=\sum A_i\varphi (x_{a_ih})=\varphi \left(\sum A_ix_{a_ih}\right)=\varphi (\alpha_h)=0.\]
If $h\in k\setminus k'$, then
\[ \Delta_{\mathcal{A}}f(0,h)=\sum A_if(a_ih)=\sum A_i\varphi (x_{a_ih})=\sum A_i\cdot 0=0,\]
by the definition of $\varphi $, since $a_ih\in k\setminus k'$ makes $x_{a_ih}\in W$. The rest follows the last part of the proof of Theorem 5. The function $f$ is measurable, since it is non-zero on a subset of the countable set $\sum k'\beta_{s_n}$.
\end{proof}

Let $n\geq 2$. If the ordinary $n$th derivative exists at $x$, so does the
ordinary $m$th derivative for each $m$, $1\leq m<n$. This property fails for
all $\mathcal{A}$-derivatives. 

\begin{corollary}
Let $\mathcal{A}$ be a generalized Riemann derivative of order $n$, $n\geq 2$%
. If $\mathcal{B}$ is any generalized Riemann derivative of order $m$ where $%
1\leq m<n$, then there is a measurable function $f$ so that $D_{\mathcal{A}%
}f\left( 0\right) $ exists, but $D_{\mathcal{B}}f\left( 0\right) $ does not.
\end{corollary}

\begin{proof}
Let $\mathcal{A=}\left\{ A_{1},\dots ;a_{1},\dots \right\} $ and $\mathcal{B=%
}\left\{ B_{1},\dots ;b_{1},\dots \right\} $. Since $m<n$, the $m$th Vandermonde conditions
force $\sum A_{i}a_{i}^{m}=0$ and $\sum B_{i}b_{i}^{m}=m!$. Consequently $%
\Delta _{\mathcal{B}}$ cannot be a linear combination of dilates of $\Delta
_{\mathcal{A}}$. Just as in the proof of Theorem 5 above, this leads to the
existence of the desired function $f$.
\end{proof}

\subsection{Proof of the main results.} We are now ready to prove Theorems 2 and 3 announced in the Introduction.

\begin{proof}[Proof of Theorem 3] Example 2 rules out the case $m>n$, and Corollary 2 rules out the case $m<n$, so $m=n$. By Theorem 5(i), $\mathcal{A}$-differentiation implies $\mathcal{B}$-differentiation is equivalent to $(\alpha )\supseteq (\beta )$. We write $\mathbf{A}=e\mathbf{A}\oplus d\mathbf{A} =\mathbf{A}^{\epsilon }\oplus \mathbf{A}^{\epsilon '}$, by (\ref{eA+dA}). The same relation yields $(\alpha )=\alpha \mathbf{A}=\alpha (e+d)\mathbf{A}=\alpha e\mathbf{A}\oplus \alpha d\mathbf{A}=(\alpha e)\oplus (\alpha d) =(\alpha^{\epsilon })\oplus (\alpha^{\epsilon '})$. A similar expression holds for $ (\beta )$. Basic ideal theory in direct sums of algebras makes the inclusion $(\alpha )\supseteq (\beta )$ equivalent to
both $(\alpha^{\epsilon })\supseteq (\beta^{\epsilon })$ and $(\alpha^{\epsilon '})\supseteq (\beta^{\epsilon '})$. These are clearly equivalent to the two desired equations.
\end{proof}

\begin{proof}[Proof of Theorem 2] The equality $m=n$ follows from Theorem 3. By Theorem 5, ``$\mathcal{A}$-differentiable $\Longleftrightarrow $ $\mathcal{B}$-differentiable'' is equivalent to ``$(\alpha )=(\beta )$'', that is to $\beta =u\alpha $, for some invertible element $u\in \mathbf{A}$.
By (AB7), we write $u=Ae_r+Bd_s$,
for $0\neq A,B\in k$ and $r,s>0$. Consequently, $\beta =Ae_r\alpha +Bd_s\alpha =Ax_re\alpha +Bx_sd\alpha $. Uniqueness of writing of $\beta $ and $\alpha $ as sums of components makes $\beta^{\epsilon }$ and $\beta^{\epsilon '}$
scalar multiples of dilates of $\alpha^{\epsilon }$ and $\alpha^{\epsilon '}$, respectively. Moreover, the equation $m=n$ makes both $\alpha $ and $\beta $ correspond to $n$th generalized Riemann differences and, by Theorem 4(i), the same is true for $\beta^{\epsilon }$ and $\alpha^{\epsilon }$. We conclude that $\Delta _{\mathcal{B}}^{\epsilon }f(x,h)$ is a scaling of $\Delta _{\mathcal{A}}^{\epsilon }f(x,h)$ and $\Delta _{\mathcal{B}%
}^{\epsilon ^{\prime }}f(x,h)$ is a non-zero scalar multiple of a dilate of $%
\Delta _{\mathcal{A}}^{\epsilon ^{\prime }}f(x,h)$.
\end{proof}

\subsection{Two examples}
Recall from Section 1 that Example 3(iii) contained an application of Theorem 3. It provided two third generalized Riemann differences $\Delta_{\mathcal{A}}$ and $\Delta_{\mathcal{B}}$ such that $\mathcal{B}$-differentiation does not imply $\mathcal{A}$-differentiation. The next example is an application of Theorem 5. It obtains the same result in a different way and additionally shows that for the same two differences $\mathcal{A}$-differentiation does not imply $\mathcal{B}$-differentiation. Similar ring- theoretic arguments can be used for many other examples.

\begin{example} Let $\Delta_{\mathcal{A}}$ and $\Delta_{\mathcal{B}}$ be the third differences of Example 3(iii). The group algebra elements corresponding to them are
\[\alpha =d_2-2d_1\text{ and }\beta =2d_{\frac 32}-6d_{\frac 12}=2x_{\frac 12}\beta',\]
where $\beta'=d_3-3d_1$. We observe that both $\alpha $ and $\beta $ are elements of $d\mathbf{A}$ since they correspond to odd differences, and $(\beta )=(\beta' )$ since the factor $2x_{\frac 12}$ is a unit in $\mathbf{A}$.
The integers $2$ and $3$ are not rational powers of each other, so they are multiplicatively $\mathbb{Q}$-linearly independent. Take the basis elements ${\lambda_1}=\log 2$ and ${\lambda_2}=\log 3$ for $\mathbb{R}$ over $\mathbb{Q}$ as in Section 4. By (AB5) we can pass from $k\mathbb{R}^+$ to $d\mathbf{A}$, and see that the elements $d_2$ and $d_3$ are algebraically independent over $k$. Since $d_1=d$ is the identity element of $d\mathbf{A}$, so too are $\alpha $ and $\beta' $. Since $\alpha $ and $\beta' $ are not multiples of each other, the ideals
\[\alpha k[\alpha, \beta' ]=(\alpha )\cap  k[\alpha, \beta' ]\text{ and }\beta' k[\alpha, \beta' ]=(\beta' )\cap  k[\alpha, \beta' ]\]
of $k[\alpha, \beta' ]=\mathbf{A}\cap  k[\alpha, \beta' ]$ are incomparable (not included in each other), and hence the same is true about the ideals $(\alpha )$ and $(\beta' )=(\beta )$ of $\mathbf{A}$. By Theorem 5(i), neither of $\mathcal{A}$-differentiability and $\mathcal{B}$-differentiability implies the other.
\end{example}

\begin{example}
Part (iii) of Example \ref{e:2} discussed $f_{\ast }\left( x\right) $, the
first order $\mathcal{A}$-derivative with excess $e=1$ associated with the
difference quotient%
\begin{equation*}
\frac{\left( \frac{1}{2}-\tau \right) f\left( x+\left( \tau +1\right)
h\right) +2\tau f\left( x+\tau h\right) -\left( \frac{1}{2}+\tau \right)
f\left( x+\left( \tau -1\right) h\right) }{h},
\end{equation*}%
where $\tau =1/\sqrt{3}$. We will now derive the result shown there. 

The derivative $f_{\ast }$ is not equivalent to ordinary differentiation. To
see this, either invoke Theorem 1 or directly test $f_{\ast }$ at $x=0$ on
the characteristic function of $\left\{ 0\right\} $. Theorem 2 allows us to
identify all $\mathcal{A}$-derivatives that are equivalent to this one.
Decompose the quotient into even and odd parts as follows.%
\begin{eqnarray*}
&&\frac{\left( \frac{1}{2}-\tau \right) x_{\left( \tau +1\right) h}+2\tau
x_{\tau h}-\left( \frac{1}{2}+\tau \right) x_{\left( \tau -1\right) h}}{h}=
\\
&&\frac{\left( \frac{1}{2}-\tau \right) \left( e_{\left( \tau +1\right)
h}+d_{\left( \tau +1\right) h}\right) +2\tau \left( e_{\tau h}+d_{\tau
h}\right) -\left( \frac{1}{2}+\tau \right) \left( e_{\left( \tau -1\right)
h}+d_{\left( \tau -1\right) h}\right) }{h}= \\
&&\frac{\Delta ^{ev}f\left(x, h\right) }{h}+\frac{\Delta ^{odd}f\left(x, h\right) 
}{h}
\end{eqnarray*}%
where 
\begin{equation*}
\Delta ^{ev}f\left(x, h\right) =\left( \frac{1}{2}-\tau \right) e_{\left( \tau
+1\right) h}+2\tau e_{\tau h}-\left( \frac{1}{2}+\tau \right) e_{\left( \tau
-1\right) h}
\end{equation*}%
and%
\begin{equation*}
\Delta ^{odd}f\left(x, h\right) =\left( \frac{1}{2}-\tau \right) d_{\left( \tau
+1\right) h}+2\tau d_{\tau h}-\left( \frac{1}{2}+\tau \right) d_{\left( \tau
-1\right) h}.
\end{equation*}%
Now apply Theorem \ref{2} to find that the most general first order $%
\mathcal{A}$-derivative equivalent to this one is associated with%
\begin{equation*}
\frac{\Delta ^{odd}f\left(x, sh\right) }{sh}+A\frac{\Delta ^{ev}f\left(
x,rh\right) }{h}
\end{equation*}%
where $s$, $r$ and $A$ are nonzero constants.
\end{example}

\end{document}